\documentclass{article}
\usepackage{amssymb,amsmath}
\usepackage{amsthm}
\usepackage{ifthen,amsfonts,graphicx,latexsym,algorithm,algorithmic,url,txfonts,bbm}
\usepackage{color}
\usepackage{verbatim}
\usepackage[latin1]{inputenc}
\usepackage[english]{babel}

\newtheorem{theorem}{Theorem}[section]
\newtheorem{lemma}[theorem]{Lemma}

\theoremstyle{remark}

\newcommand{\ve}{\varepsilon}

\newcommand{\pical}{\mathcal{P}}

\newcommand{\R}{\mathbb{R}}

\newcommand{\Lip}{\mathrm{Lip} }

\newcommand{\spt}{\mathrm{spt}}

\graphicspath{{./MFGRESULTS/}}

\def\og{\leavevmode\raise.3ex\hbox{$\scriptscriptstyle\langle\!\langle$~}}
\def\fg{\leavevmode\raise.3ex\hbox{~$\!\scriptscriptstyle\,\rangle\!\rangle$}}

\title{Lipschitz estimates on the JKO scheme for the Fokker-Plack equation on bounded convex domains}
\author{Vincent Ferrari\thanks{\'Ecole Normale Sup\'erieure; 45 Rue d'Ulm, 75005 Paris, FRANCE, email: \textsf{vferrari@clipper.ens.fr}}$\,$ and Filippo Santambrogio\thanks{Institut Camille Jordan, Universit\'e Claude Bernard - Lyon 1; 43 boulevard du 11 novembre 1918,
69622 Villeurbanne cedex,
FRANCE, email: \textsf{santambrogio@math.univ-lyon1.fr} } \thanks{Institut Universitaire de France (IUF)}}

\begin{document}
\maketitle

\begin{abstract} Given a semi-convex potential $V$ on a convex and bounded domain $\Omega$, we consider the Jordan-Kinderlehrer-Otto scheme for the Fokker-Planck equation with potential $V$, which defines, for fixed time step $\tau>0$, a sequence of densities $\rho_k\in\pical(\Omega)$. Supposing that $V$ is $\alpha$-convex, i.e. $D^2V\geq \alpha I$, we prove that the Lipschitz constant of $\log\rho+V$ satisfies the following inequality: $\Lip(\log(\rho_{k+1})+V)(1+\alpha \tau)\leq \Lip(\log(\rho_{k})+V)$. This provides exponential decay if $\alpha>0$, Lipschitz bounds on bounded intervals of time, which is coherent with the results on the continuous-time equation, and extends a previous analysis by Lee in the periodic case.
\end{abstract}
\section{Introduction}

This short paper is concerned with the parabolic PDE 
\begin{equation}\label{FP}
\partial_t\rho-\Delta\rho-\nabla\cdot(\rho \nabla V)=0,
\end{equation}
which is known as Fokker-Planck equation. It includes a linear diffusion effect (corresponding to the Laplacian in the above equation) and advection by a drift, which is supposed to be  of gradient type. The function $V$ is supposed here to be given, independent of time and of $\rho$, and sufficiently smooth. We consider this equation on a bounded domain $\Omega$, and we complete it with no-flux boundary conditions
$$(\nabla\rho+\rho\nabla V)\cdot \mathbf{n}=0.$$ 

Starting from the seminal paper \cite{JKO}, it is well-known that this equation is the gradient flow of the functional
$$\rho\mapsto\mathcal E(\rho):=\int \rho\log\rho + \int Vd\rho$$
with respect to the Wasserstein distance $W_2$. This distance is the one defined from the minimal cost in the optimal transport  with quadratic cost. The reader can refer to \cite{villani,OTAM} for more details about optimal transport and Wasserstein distances, and \cite{AmbGigSav,surveyGradFlows} for gradient flows in this setting.

The same paper \cite{JKO} also provides a time-discretized approach to the above equation, now known as the Jordan-Kinderlehrer-Otto scheme (JKO). It consists in fixing a time step $\tau>0$ and then iteratively solving
$$\min\left\{  \frac{W_{2}^{2}(\rho,\rho_k)}{2\tau} +\mathcal E(\rho)\,:\,\rho\in\pical(\Omega)\right\},$$
defining $\rho_{k+1}$ as the unique minimizer of the above optimization problem. The sequence obtained in this way represents a good approximation of the continuous-time solution, in the sense that $\rho_k$ is close to $\rho(k\tau,\cdot)$.

For this equation, which is linear and quite simple, using the JKO scheme as a tool to approximate the solution (for numerical or theoretical purposes) is not necessary, but the same scheme can be useful for equations which are gradient flow of more involved functionals. In order to prove convergence, it is often necessary to prove compactness estimates on the solutions of the scheme. Among the bounds that one can prove, the most striking ones are those where a certain norm decreases when passing from $\rho_k$ to $\rho_{k+1}$. The case where a same norm increases when passing from $\rho_k$ to $\rho_{k+1}$, but no more than a quantity of the order of $\tau$, so that the bound can be iterated, is also interesting. We cite for instance $L^\infty$ bounds (see for instance Section 7.4.1 in \cite{OTAM}), BV bounds (see \cite{demesave}), which are valid when $V=0$ and the entropy $\int \rho\log\rho$ is replaced by an arbitrary convex penalization $\int f(\rho)$. 

We are interested in this paper in Lipschitz bounds and more precisely in the quantity $\Lip(\log\rho+V)$. Using for instance the well-known Bakry-Emery theory (see \cite{BakGenLed}) this quantity decreases in time along solutions of the continuous equation (which is a re-writing of the transport-diffusion equation $\partial_t u = \Delta u -\nabla V\cdot \nabla u$ if one sets $u=\rho e^V$) if $V$ is convex. If $D^2V\geq \alpha I$ one can see that $e^{\alpha t}\Lip(\log\rho_t+V)$ decreases in time, which implies exponential convergence as $t\to \infty$ if $\alpha>0$ or controlled growth if $\alpha<0$.

It is highly remarkable that these bounds can exactly be translated into the JKO setting, i.e. they also holds along the discrete-time iterations of the scheme. This was already observed in \cite{LeeJKO}, and the analysis is very similar to ours, but was only restricted to the periodic setting. The main object of the present paper is to extend the analysis of \cite{LeeJKO} to a case with boundary. This recalls what is done in \cite{DiMSan}, where a similar estimate is performed on the quantity 
$$\int \left|\frac{\nabla\rho}{\rho}+\nabla V\right|^pd\rho,$$
which is proven to decrease along iterations if $V$ is convex and the domain is also convex (this was obtained thanks to the so-called {\it five gradients inequality} introduced in \cite{demesave}). Taking the power $1/p$ and sending $p\to\infty$ also provides an estimate on $||\frac{\nabla\rho}{\rho}+\nabla V||_{L^\infty}$. The same computation unfortunately does not work if we only have $D^2V\geq \alpha I$ with $\alpha<0$, since a similar estimate is only possible as soon as $1+\alpha\tau p>0$, which prevents to send $p\to \infty$
without sending $\tau\to 0$. In this sense, it is not an estimate on the JKO scheme.
Moreover, obtaining an $L^\infty$ bound (i.e. an estimate on a maximal value) out of a limit of integral quantities and using the five gradients inequality is a complicated procedure which can be highly simplified, which is the scope of this note.

As in \cite{LeeJKO}, we will obtain the desired estimate by taking the maximal point of the squared norm of a gradient and using the optimality conditions together with the Monge-Amp\`ere equation on the optimal transport map from $\rho_{k+1}$ to $\rho_k$. Even if we believe that the presentation that we give here is lighter and if the regularity assumptions on the date have been reduced to the minimal ones, the computations that we present are exactly the same as those of  \cite{LeeJKO} in the case of a torus. The added value of our work is the attention to the boundary. The key point are
\begin{itemize}
\item a lemma which guarantees that $\max |\nabla\varphi|$ is not attained on the boundary when $\varphi$ is the Kantorovich potential between two smooth densities on a ball;
\item an extension-approximation procedure to pass from an arbitrary convex domain to a larger ball, and to regularize the densities.
\end{itemize}

\section{Main result}
Let $\Omega$ be a bounded closed convex domain in $\mathbb{R}^{d}$, V a non-negative Lipschitz function on $\Omega$, $\tau >0$ and $g$ a probability density on $\Omega$ (i.e. $g\in L^1(\Omega), g\geq 0$ and $\int _\Omega g=1$). We study the minimizer of the functional $\mathcal{F}^\tau_{g,V}$ on $\mathcal{P}(\Omega)$ (the set of probabilities measures on $\Omega$) given by :

\begin{equation}
	\mathcal{F}^\tau_{g,V}(\rho) = \frac{W_{2}^{2}(\rho,g)}{2\tau} +\mathcal E(\rho)
= \frac{W_{2}^{2}(\rho,g)}{2\tau} + \int_{\Omega} \rho \log \rho + \int_{\Omega} V d\rho
\end{equation}
where $W_{2}$ is the Wasserstein distance on $\mathcal{P}(\Omega)$.

We know that this functional admits a unique minimizer, denoted $\rho$ in the rest of the paper. Indeed, $\mathcal F$ is the sum of three convex functionals, and the entropy part is strictly convex, which provides uniqueness. We denote $T(x) = x -\nabla \varphi(x)$ the optimal transport map between $\rho$ and $g$ (which is the gradient of the convex function $x\mapsto \frac{|x|^2}{2}-\varphi(x)$, see \cite{Brenier polar}), where $\varphi$ is the corresponding Kantorovitch potential. We recall that the minimizer $\rho$ satisfies the following properties (we refer to Chapter 8 in \cite{OTAM}):
\begin{gather}
    \rho \mbox{ is Lipschitz and strictly positive on }\Omega \notag\\
	 \log \rho + V + \frac{\varphi}{\tau} = \mbox{constant on }\Omega \label{optimality}\\
	\mathrm{det}(I - D^{2}\varphi(x)) = \frac{\rho(x)}{g(T(x))} \label{monge-ampere}
\end{gather} 
Equation \eqref{monge-ampere} is the Monge-Amp\`ree equation for the transport between $\rho$ and $g$ and \eqref{optimality} is the optimality condition for $\rho$.
We will prove the following:

\begin{theorem}
Suppose that $\Omega\subset\R^d$ is a bounded convex set, that $g$ is bounded from below by a strictly positive constant and Lipschitz continuous on $\Omega$, and that $V$ is Lipschitz continuous and uniformly $\alpha$-convex on the same domain for some $\alpha\in\R$. Then we have 
\begin{align}
	Lip_{\Omega}(\log \rho + V) (1+ \alpha \tau) &\leqslant Lip_{\Omega}(\log g + V)
\end{align}
where $Lip_{\Omega}(f) = \underset{x \neq y \in \Omega}{sup} \frac{|f(x) - f(y)|}{|x-y|}$
\end{theorem}

\section{The proof}

\begin{proof}

We follow the same idea as in \cite{LeeJKO} (Lemma 3.2), except that we need to get rid of a possible maximum on the boundary We eliminate this case when the domain is a ball and we do the general case by approximation.

\subsection{The computation}

Suppose that $V$ and $\log g$ are $C^{1,\beta}$ functions for some $\beta\in (0,1)$ and that $\Omega$ is smooth and uniformly convex. In this case, the regularity theory for the Monge-Ampere equation by Caffarelli (\cite{caf3, DePFig survey, Fig book}) provides $\varphi\in C^{3,\beta}$. Indeed, we first observe that $\rho$ is Lipschitz and bounded from below, so we can apply Caffarelli's theory to obtain $\varphi\in C^{2,\beta}$ (using $g,\rho\in C^{0,\beta}$ and the assumptions on the boundary of $\Omega$). This implies $\log\rho\in C^{1,\beta}$ and allows to obtains $\varphi\in C^{3,\beta}$, which is the regularity we need for our next computations.  

We proceed by a maximum principle argument : let $x_{0}$ be a point where $|\nabla\varphi|^{2}$ achieves its maximum. If $x_{0}$ is in the interior of $\Omega$, the optimality conditions are :
\begin{align} 
    \sum_{i=1}^{d} \varphi_{ij}\varphi_{i}(x_{0}) &= 0 \mbox{  for all } j \label{optimality-max}\\
	(\varphi_{ijk} \varphi_{i} + \varphi_{ij} \varphi_{ik})_{(j,k)}(x_{0}) &\leqslant 0 \mbox{  for all }  i \label{optimality-max-2nd}
\end{align}
where the indices correspond to partial derivatives in the associated direction and the inequality in the second line means that the matrix (indexed by $j$ and $k$) is negative semidefinite for every $i$.

We differentiate the logarithm of $\eqref{monge-ampere}$ in the $i$ direction to obtain :
\begin{equation}
    -\mathrm{Tr}(B^{jk}\varphi_{ijk}(x_{0})) = (\log(\rho))_{i}(x_{0}) - (\log g)_{i}(T(x_{0})) + \sum_{k=1}(\log g)_{k} (T(x_{0})) \varphi_{ik}(x_{0}) \label{monge_amp_der}
\end{equation}
where the matrix $B^{jk}$ is given by $B= (I-D^{2}\varphi(x_{0}))^{-1}$.

By using that $B$ is a symmetric and non-negative matrix (as it is the inverse of the Hessian of a convex function) and the optimality conditions \eqref{optimality-max-2nd}, we get~:
\begin{equation}\label{Bijkfii}
	-\mathrm{Tr}(B^{jk}\varphi_{ijk}\varphi_{i}) \geqslant \mathrm{Tr}(B^{jk}\varphi_{ij}\varphi_{ik}) \geqslant 0
\end{equation}
where we use $A \geqslant 0, B \leqslant 0\Rightarrow \mathrm{Tr}(AB) \leqslant 0$, with $A=(D^2\varphi)^2\geq 0$.
Then, by multiplying \eqref{monge_amp_der} by $\varphi_{i}$ and using \eqref{optimality-max},\eqref{Bijkfii},
\begin{equation}\label{nabna}
	\nabla\varphi(x_0) \cdot \nabla (\log \rho) (x_{0})  \geqslant \nabla (\log g) (T(x_{0})) \cdot \nabla \varphi (x_{0})
\end{equation}

 Equation \eqref{optimality} gives  $\nabla (\log \rho) =- \nabla V - \frac{\nabla \varphi}{\tau}$ and thus we get :
\begin{align}
	\frac{|\nabla \varphi(x_0)|^{2}}{\tau} \leqslant \nabla \varphi(x_0) \cdot \big(-\nabla V(x_{0}) &-(\nabla (\log g))(T(x_{0}))\big) \label{introduction_u} \notag\\
	\frac{|\nabla \varphi(x_0)|^{2}}{\tau} + \nabla \varphi(x_0) \cdot \big( \nabla V(x_{0}) - \nabla V(T(x_{0}))\big) &\leqslant -\nabla \varphi(x_0) \cdot (\nabla (V+\log g)(T(x_{0})) \notag\\
	\frac{|\nabla \varphi(x_0)|^{2}}{\tau} + (x_{0} - T(x_{0})) \cdot (\nabla V (x_{0}) - \nabla V(T(x_{0})) &\leqslant |\nabla \varphi(x_0)|~ \Vert \nabla (V \log g) \Vert_{\infty} \notag\\
	\frac{|\nabla \varphi(x_0)|^{2}}{\tau}(1 + \tau \alpha ) &\leqslant |\nabla \varphi(x_0)|~ \Vert \nabla (V +\log g) \Vert_{\infty}\\
\end{align}
The last expression can be simplified in order to obtain
$$\Vert\nabla (V + \log \rho) \Vert_{\infty} (1+ \alpha \tau)= \frac{|\nabla \varphi(x_0)|}{\tau}(1 + \tau \alpha ) \leqslant \Vert \nabla (V + \log g) \Vert_{\infty}$$
since $x_{0}$ is the point where $|\nabla \varphi|^{2}$ and also $|\nabla (\log \rho + V)|$ achieve their maximal values. We used the fact that, if V is $\mathcal{C}^{1}$ and $\alpha$ uniformly convex, then for all $x, y$ we have $(\nabla V(x) - \nabla V(y) )\cdot (x-y) \geqslant \alpha |x-y|^{2}$.

\subsection{Treating the boundary}

The computations above required the maximum point $x_{0}$ to be an interior point of $\Omega$. We will remove this difficulty in two steps: first we show that if $\Omega$ is a ball the maximum can't be achieved on the boundary. Then we reduce the general case to this one by approximation.

\subsubsection{If the domain is a ball}

In this subsection, we suppose $\Omega = B_{R}(0)$ and prove the following
\begin{lemma}
Suppose that $\Omega = B_{R}(0)$ is a ball and that $\varphi$ is the Kantorovich potential for the quadratic cost between two densities $\rho,g\in C^{0,\beta}(\Omega)$ which are both 
bounded from below. Then $\phi\in C^1$ and the maximum of $ |\nabla \varphi|$ cannot be achieved on $\partial\Omega$. \end{lemma}
\begin{proof}
Let us define $f(x) = \frac{1}{2} |\nabla \varphi(x)|^{2}$. From the regularity on the densities and on the domain, Caffarelli's theory implies $\varphi\in C^{2,\beta}$ and $T$ is a homeorphism of $\Omega$. Therefore it sends the boundary to the boundary (an observation already used for similar purposes in Lemma 2.4 in \cite{iacobelli2019weighted}). Now, for $x \in  \partial \Omega$ we have
\begin{equation*}
	|T(x)|^{2} = |x-\nabla \varphi(x)|^{2} = R^{2}\quad \Rightarrow\quad
	f(x) = x\cdot \nabla \varphi(x)
\end{equation*}

Suppose that $f$ achieve its maximum in $\Omega$ at $x_{0} \in \partial \Omega$. Since $f(x) = x\cdot \nabla \varphi$ in $\partial \Omega$, $x\cdot \nabla \varphi$ also achieve its maximum in $\partial \Omega$ at $x_{0}$. If we denote $\mathbf{n}(x) = \frac{x}{R}$ the normal at $x\in \partial \Omega$, there exists $\lambda \in \mathbb{R}^{+}$ (global maximum) and $\mu \in \mathbb{R}$ (maximum on the boundary) so that :
\begin{align}
	D^{2}\varphi(x_{0}) \cdot \nabla \varphi(x_{0}) &= \lambda \mathbf{n}(x_{0}) \label{boule_1}\\
	D^{2}\varphi(x_{0}) \cdot x_{0} + \nabla \varphi(x_{0}) &= \mu \mathbf{n}(x_{0})	  \label{boule_2}\
\end{align}  

By multiplying equation \eqref{boule_2} by $\mathbf{n}(x_{0})=\frac{x_0}{R}$, we get :
\begin{equation}
	R \mathbf{n}(x_{0})\cdot D^{2}\varphi(x_{0}) \cdot  \mathbf{n}(x_{0}) + \frac{f(x_{0})}{R} = \mu\quad\Rightarrow\quad
	\frac{f(x_{0})}{R}  > \mu - R \label{boul_ineg_1}
\end{equation}
where we used the inequality $D^{2}\varphi < I$, which is strict since the Monge-Ampere equation \eqref{monge-ampere} tells us that $I - D^{2}\varphi$ can't have a zero eigen-value, since $\rho >0$. 
If we multiply the same equation \eqref{boule_2} by $\nabla \varphi(x_{0})$ and use  \eqref{boule_1}, we get :
 $$\nabla \varphi(x_{0}) \cdot 	D^{2}\varphi(x_{0}) \cdot x_{0}+ |\nabla \varphi(x_{0})|^{2} = \mu \mathbf{n}(x_{0}) \cdot \nabla \varphi (x_{0})\quad\rightarrow\quad
	\lambda x_{0} \cdot \mathbf{n}(x_{0}) + 2f(x_{0}) = \mu \frac{x_{0}}{R} \cdot\nabla \varphi(x_{0}),
	$$
	which can be re-written as 
\begin{equation}
	f(x_{0})(2 - \frac{\mu}{R}) = - R \lambda \leqslant 0\quad \Rightarrow\quad	\mu \geqslant 2R, \label{boule_ineg_2}
\end{equation}

since $\lambda \geqslant 0$. Putting \eqref{boul_ineg_1} and \eqref{boule_ineg_2} together, we obtain $	f(x_{0}) > R^{2}$.

However we know that if $\Omega$ is a strictly convex domain, the transport plan shall verify for all $x\in \partial \Omega,$ $\mathbf{n}(T(x))\cdot~\mathbf{n}(x)\geqslant0$ (see the proof of Lemma 2.4 in \cite{iacobelli2019weighted}). In the case of a ball this implies that $T(x)$ is in the same hemisphere as $x$, thus in term of the distance between the two points we have $|T(x) - x| = |\nabla \varphi(x)| \leqslant \sqrt{2} R$. Therefore we have $f(x_{0}) \leqslant R^{2}$ on $\partial \Omega$, which leads to a contradiction if the maximum is attained on the boundary.\end{proof}

\subsubsection{Approximation}
We come back to the general case and we fix $R>0$ such that $\Omega \subset B_{R}$. The idea is to approximate our problem by a sequence of regularized problems defined on $B_{R}$ with a sequence of potentials $V_n$ diverging to infinity out of $\Omega$. First we extend $V$ to $\mathbb{R}^{d}$ by defining a function $\tilde{V}$ :
\begin{equation*}
    \tilde{V}(x) = \underset{y \in \Omega}{\sup} \left( V(y) + \nabla V(y) \cdot (x-y) + \frac{\alpha}{2} |x-y|^{2} \right) 
\end{equation*}
This extension is non-negative and uniformly $\alpha$-convex in $\mathbb{R}^{d}$. It coincides with $V$ on $\Omega$ since $V$ itself is uniformly $\alpha$-convex.
We consider now the function $h:= \log g +V$ and extend it to $\mathbb{R}^{d}$ by defining $\tilde h$ :
\begin{align*}
    \tilde h (x) = \underset{y \in \Omega}{\sup}\left(h(y) + \Lip(h) |x-y| \right)
\end{align*}
where $\Lip(h)$ is the Lipschitz constant of $h$. We have $\Lip(h) = \Lip(\tilde h)$.
We then take convolutions in order to enforce higher regularity. Let $\xi\in C^{\infty}(\mathbb{R}^{d},\mathbb{R^{+}})$ be a convolution kernel, supported in $B_{1}$, and satisfying $\int_{\mathbb{R}^{d}} \xi dx= 1$. We then rescale $\xi$ defining $ \xi_{n}(x) = n^{d} ~ \xi(n x)$ and set
\begin{align}
	V_{n}(x) &= \int_{R^{d}} (\tilde V(y) + n d(y,\Omega)^{2}) \xi_{n}(x-y) dy \notag\\
	h_{n}(x) &= \int_{R^{d}} \tilde h(y)  \xi_{n}(x-y) dy \notag\\
	g_{n}(x) &= \lambda_{n} e^{h_{n}(x) - V_{n}(x)}\notag
\end{align}
where $\lambda_{n}$ is a normalisation constant chosen so as to ensure $\int_{B_{R}} g_{n} = 1$. We note that for every~$n$ we have $\Vert \nabla h_{n} \Vert_{\infty} = \Lip(h_{n}) \leqslant \Lip(\tilde h) = \Lip(h) $ and that $V_{n}$ is $\alpha$-uniformly convex and non-negative in $\mathbb{R}^{d}$.
We also define a sequence of functional $\mathcal{F}_{n}$ on $\mathcal{P}(B_{R})$ :
\begin{equation}
	\mathcal{F}_{n}(\rho) := \mathcal{F}^\tau_{g_n,V_n}(\rho) = \frac{W^{2}_{2}(\rho,g_{n})}{2\tau} + \int_{B_{R}} \rho \log \rho + \int_{B_{R}} V_{n} d\rho
\end{equation}

By applying our previous computations to the minimizer $\hat\rho_n$ of the functional $\mathcal{F}_{n}$ (note that the regularity is enforced by the convolution and that we are on the ball, so that the maximum of the squared gradient of the Kantorovich potential cannot be achieved on the boundary) we obtain
\begin{align}
    \Vert\nabla (V_{n} + \log \tilde \rho_{n}) \Vert_{L^\infty(B_R)} (1+ \alpha \tau) &\leqslant \Vert \nabla (V_{n} + \log g_{n}) \Vert_{L^\infty(B_R)} \leqslant \Lip_{\Omega} (V + \log g).\label{resultat_suite}
\end{align}

We now study the convergence of the $V_{n}$, $f_{n}$, $g_{n}$ and $\mathcal{F}_{n}$. By the properties of the convolution and since $\tilde{h}$ is extension of $h$, $h_n$ converges uniformly to $h$ in $\Omega$. As for $V_n$, the situation is trickier, because we added a penalization of the form $n d(\cdot,\Omega)^{2}$ before convolving. Since the support of $\xi_n$ is contained in the ball of radius $1/n$, when $x\in\Omega$ this penalization contributes at most $n.(1/n)^2$ to the integral, so we have $\tilde V*\xi_n\leq V_n\leq  \tilde V*\xi_n+1/n$ on $\Omega$. Hence $V_n$ also converges uniformly to $ V$ on $\Omega$. On the other hand, for all $x \in B_{R} \setminus \Omega$, we have $\underset{n \to \infty}{\lim} V_{n}(x) = + \infty$. We deduce that $g_{n}$ converge uniformly to $g$ in $\Omega$ and pointwise to $0$ in $B_{R} \setminus \Omega$. Since $g_n$ is a probability measure, we also have the weak convergence as probability measures of $g_n$ to the measure supported on $\Omega$ with density $g$.

Then, a quick computation provides $\Lip(V_{n})_{\Omega} \leqslant \Lip(V)_{\Omega} + 2$ (the additional term $+2$ comes from the gradient of the penalization $\nabla \big(nd(\cdot, \Omega)^{2}\big)=2nd(\cdot, \Omega)\nabla d(\cdot, \Omega)$, whose modulus is almost $2$ at points of $\Omega+\spt(\xi_n)$), thus by \eqref{resultat_suite} we get:
\begin{align}
    \Lip(\log \rho_{n})_{\Omega} \leqslant \frac{\Lip(V + \log g)_{\Omega} }{1 + \alpha \tau} + \Lip(V)_{\Omega} + 2. \label{borne_lip_log_rho} 
\end{align}
This bound, together with the mass constraint $\int_\Omega \rho_n\leq \int_{B_R}\rho_n=1$ implies a uniform upper bound on $\rho_n$. Composing with the exponential function, which is Lipschitz continuous on bounded intervals, we deduce that $\rho_{n}$ is uniformly Lipschitz in $\Omega$. 
We will now characterize the limit of $\rho_{n}$. 

The sequence of functionals $\mathcal{F}_{n}$ $\Gamma$-converges for the weak topology on measures to the functional $\mathcal{F}$ defined as
$$\mathcal F(\rho):=\begin{cases} \mathcal F^\tau_{g,V}(\rho) &\mbox{ if }\spt(\rho)\subset\Omega,\\
							+\infty &\mbox{ if not.}\end{cases}$$
Indeed, if the sequence $\rho_{n} \in \mathcal{P}(B_{R})$ weakly converges to $\rho \in \mathcal{P}(\Omega)$, we have
\begin{align}
    \mathcal{F}(\rho) \leqslant \underset{n \to \infty}{\liminf}~ \mathcal{F}_{n}(\rho_{n}).
\end{align}
This can be proven in the following way:  $\rho \mapsto \int \rho \log \rho $ is lower-semicontinous for the weak convergence of measure and  $W_{2}$ metrizes the same convergence; if $\int V_n d\rho_n\leq C$ then $\rho_n(\{x\,:\, d(x,\Omega)>\ve\})\to 0$ as $n\to \infty$, hence $\spt(\rho)\subset\Omega$; we also have $\int V_n d\rho_n\geq \int (\tilde V*\xi_n)d\rho_n\to \int \tilde Vd\rho=\int Vd\rho$ where the last equality is valid if $\rho$ is concentrated on $\Omega$.

On the other hand, if $\rho \in \mathcal{P}(\Omega)$, then the constant sequence $\rho_{n} = \rho$ satisfies
\begin{align}
    \underset{n \to \infty}{\limsup}~ \mathcal{F}_{n}(\rho_{n}) \leqslant \mathcal{F}(\rho)
\end{align}

We conclude that the optimizers $\tilde\rho_{n}$ weakly converge to $\rho$, the only minimizer of $\mathcal{F}$. Since $\tilde\rho_{n}$ is uniformly Lipschitz in $\Omega$, we deduce that the convergence is uniform in $\Omega$. Finally taking the limit of \eqref{resultat_suite} we obtain
\begin{align}
    \Lip_{\Omega}(\log \rho + V) (1+ \alpha \tau) &\leqslant \Lip_{\Omega}(\log g + V).\qedhere
\end{align}
\end{proof}
%
%
%
%
%

{\bf Acknowledgments} The second author acknowledges the support of the Monge-Amp\`ere et G\'eom\'etrie Algorithmique project, funded by Agence nationale de la
recherche (ANR-16-CE40-0014 - MAGA).


\begin{thebibliography}{99}


\bibitem{AmbGigSav}  {\sc L. Ambrosio, N. Gigli, G. Savar\'e}, {\it Gradient flows in metric spaces and in the space of probability measures}, Lectures in Mathematics, Birkh\"auser (2005).


\bibitem{BakGenLed} {\sc D. Bakry, I. Gentil and M. Ledoux}, {\it Analysis and Geometry of Markov Diffusion Operators}, Springer, 2014 


%
%
\bibitem{Brenier polar} {\sc Y. Brenier}, D\'ecomposition polaire et r\'earrangement monotone des champs de vecteurs. (French) {\it C. R. Acad. Sci. Paris S\'er. I Math.} 305 (19), 805--808, 1987.
%
%
%
%
\bibitem{caf3} 
{\sc L. Caffarelli}, Some regularity properties of solutions of Monge Amp\`ere equation, {\it Comm. Pure Appl. Math.} 44, 965-969, 1991.


\bibitem{DePFig survey} 
{\sc G. De Philippis and A. Figalli}, The Monge--Amp\`ere equation and its link to optimal transportation, {\it Bulletin of the American Mathematical Society} 51, 527-580, 2014.

\bibitem{demesave} 
{\sc G. De Philippis, A. R. M\'esz\'aros, F. Santambrogio, B. Velichkov},
BV estimates in optimal transportation and applications,
Archive for Rational Mechanics and Analysis 219 (2), 829-860, 2016.
%

\bibitem{DiMSan}  {\sc S. Di Marino F. Santambrogio}
JKO estimates in linear and non-linear Fokker-Planck equations, and Keller-Segel: $L^p$ and Sobolev bounds. Preprint, 2019.



\bibitem{Fig book} 
{\sc A. Figalli}, {\it The Monge-Amp\`ere Equation and Its Applications},
Zurich Lectures in Advanced Mathematics, European Mathematical Society (EMS), Z\"urich, 2017.



\bibitem{iacobelli2019weighted} {\sc M. Iacobelli, F. S. Patacchini, F. Santambrogio} Weighted ultrafast diffusion equations: from well-posedness to long-time behaviour, {\it Arch. Rati. Mech. An.}, 232 (3), 1165-1206, 2019.



\bibitem{JKO} {\sc R. Jordan, D. Kinderlehrer, F. Otto.} The variational formulation of the Fokker-Planck equation. {\it SIAM J. Math. An.}, 1998.

 \bibitem{LeeJKO} {\sc P. Lee,} On the Jordan-Kinderlehrer-Otto scheme, {\it J. Math. An. Appl.}    Vol 429, Issue 1, 131-142, 2015.

 \bibitem{LeeJKO-2} {\sc P. Lee,} A Harnack inequality for the Jordan-Kinderlehrer-Otto scheme, {\it
Journal of Evolution Equations} volume 18, pages 143--152 (2018)

\bibitem{OTAM} {\sc F. Santambrogio} {\it Optimal Transport for Applied Mathematicians}, book, {\it Progress in Nonlinear Differential Equations and Their Applications}  87, Birkh\"auser
Basel (2015).


\bibitem{surveyGradFlows} {\sc F. Santambrogio}
$\{$Euclidean, Metric, and Wasserstein$\}$ Gradient Flows: an overview, {\it Bull. Math. Sci.}, Vol 7 (1),  87--154, 2017.


\bibitem{villani} {\sc C. Villani} {\it Topics in Optimal Transportation}. Graduate Studies in Mathematics, AMS, (2003).
 

\end{thebibliography}
\end{document}